\documentclass[12pt]{amsart}
\usepackage[]{amsthm}
\usepackage[]{amsmath}
\usepackage[]{geometry}
\usepackage[]{hyperref}

\newtheorem{theorem}{Theorem}
\newtheorem*{theorem*}{Theorem}
\newtheorem{lemma}{Lemma}

\theoremstyle{definition}

\newtheorem{definition}{Definition}

\title[Automatic Conjecture and Proof of Families of Continued Fractions]
{Automatic Conjecturing and Proving of Exact Values of Some Infinite Families
of Infinite Continued Fractions}

\date{\today}

\author{Robert Dougherty-Bliss}
\address{Department of Mathematics, Rutgers University (New Brunswick), Hill Center-Busch Campus, 110 Frelinghuysen Rd., Piscataway, NJ 08854-8019, USA}
\email{robert.w.bliss@gmail.com}

\author{Doron Zeilberger}
\address{Department of Mathematics, Rutgers University (New Brunswick), Hill Center-Busch Campus, 110 Frelinghuysen Rd., Piscataway, NJ 08854-8019, USA}
\email{doronzeil@gmail.com}

\begin{document}

\begin{abstract}
    Inspired by the recent pioneering work, dubbed ``The Ramanujan Machine'' by
    Raayoni et al.~\cite{ramanujan}, we (automatically) [rigorously] prove some
    of their conjectures regarding the exact values of some specific infinite
    continued fractions, and generalize them to evaluate infinite families
    (naturally generalizing theirs). Our work complements their beautiful
    approach, since we use {\it symbolic} rather, than {\it numeric}
    computations, and we instruct the computer to not only discover such
    evaluations, but at the same time prove them rigorously.
\end{abstract}

\maketitle

\bigskip
\qquad { \it In fond memory of Richard A. Askey (1933-2019) who taught us that Special Functions are Useful Functions.}
\bigskip

\section{Introduction}
\label{sec:introduction}

\emph{Continued fractions}, possibly finite expressions of the form
\begin{equation*}
    a_0 + \cfrac{b_1}{a_1 + \cfrac{b_2}{a_2 + \cfrac{b_3}{a_3 + \cdots}}}
    =
    a_0 + \frac{b_1|}{|a_1} + \frac{b_2|}{|a_2} + \frac{b_3|}{|a_3} + \cdots,
\end{equation*}
have been employed since antiquity. Aristarchus of Samos (b.~310 BCE), the
great Ancient Greek astronomer and first-rate mathematician, argued that $r$,
the radius of the Sun divided by the radius of Earth, satisfies
\begin{equation*}
    \frac{r}{r - 1} > \frac{71755875}{61735500}.
\end{equation*}
To provide a simpler bound, he replaced $71755875 / 61735500$ with $43 / 37$
and claimed that the same bound would hold. Continued fractions lurk here; the
later fraction appears in the former's continued fraction expansion:
\begin{align*}
    \frac{43}{37} &= 1 + \frac{1|}{|6} + \frac{1|}{|6} \\
    \frac{71755875}{61735500} &= 1 + \frac{1|}{|6} + \frac{1|}{|6} + \frac{1|}{|4} + \frac{1|}{|1} + \frac{1|}{|2} + \frac{1|}{|1} + \frac{1|}{|2} + \frac{1|}{|1} + \frac{1|}{|6}.
\end{align*}
Thus Aristarchus and the other Greeks certainly knew \emph{something} about
continued fractions, though the full extent of their knowledge is lost to
history (see \cite[p.~336]{heath} and \cite[ch.~1]{brezinski}).

The Italian Pietro Cataldi (b.~1548) was the first mathematician to expound a
proper theory and notation for continued fractions. He noted that continued
fractions, when they converge, alternate between being above and below their
limits, and essentially gave us the notation
\begin{equation*}
    a_0 + \frac{b_1|}{|a_1} + \frac{b_2|}{|a_2} + \frac{b_3|}{|a_3} + \cdots
\end{equation*}
which we employ here.

Cataldi may have been the first to present the theory of continued fractions,
but the great mathematician Leonhard Euler was responsible for the theory's
explosion and widespread application. In \emph{Introductio in Analysin
Infinitorum}, his influential analysis textbook, Euler treated continued
fractions extensively and showed that they are in a sort of correspondence with
infinite series. (Every series can be represented as a continued fraction, and
every continued fraction can be represented as a series.) Based on this
reasoning, he gave some of the first examples of continued fractions for famous
constants:
\begin{align*}
    e &= \frac{1|}{|2} + \frac{1|}{|1} + \frac{1|}{|2} + \frac{1|}{|1} + \frac{1|}{|1} + \frac{1|}{|4} + \frac{1|}{|1} + \frac{1|}{|1} + \frac{1|}{|6} + \cdots \\
    \frac{e + 1}{e - 1} &= \frac{1|}{|2} + \frac{1|}{|6} + \frac{1|}{|10} + \frac{1|}{|14} + \frac{1|}{|18} + \frac{1|}{|22} + \cdots
\end{align*}
The most well-known general case of Euler's results is now called \emph{Euler's
continued fraction}:
\begin{equation*}
    \cfrac{1}{1 - \cfrac{r_1}{1 + r_1 - \cfrac{r_2}{1 + r_2 - \cfrac{r_3}{1 + r_3 - \cdots}}}}
    = \sum_{k \geq 0} \prod_{j = 1}^k r_j.
\end{equation*}

Such identities are intrinsically fascinating, but continued fraction
expansions have found wide applications in number-theoretic irrationality
proofs. There is always hope that the correct continued fraction will provide a
Diophantine approximation sufficiently nice to prove the irrationality of a
famous constant, \emph{\'a la} Roger Ap\'ery's proof that $\zeta(3)$ is
irrational (see \cite{poorten}).

In this paper, we would like to combine the spirits of Euler and Aristarchus to
present an experimental method for automatically discovering (and proving!)
continued fraction expansions. We were inspired by ``The Ramanujan Machine,'' a
recent \emph{inverse symbolic calculator} that numerically conjectures
continued fraction expansions involving well-known constants \cite{ramanujan}.
For example, one of their conjectures is
\begin{equation}
    \label{example-frac}
    \frac{e}{e - 2} = 4 - \frac{1|}{|5} - \frac{2|}{|6} - \frac{3|}{|7} - \cdots = 4 - \cfrac{1}{5 - \cfrac{2}{6 - \cfrac{3}{7 - \cdots}}}.
\end{equation}
Another is
\begin{equation*}
    \frac{1}{e - 2} = \frac{1|}{|1} - \frac{1|}{1|} + \frac{2|}{|1} - \frac{1|}{1|} + \frac{3|}{1|} - \frac{1|}{|1} + \cdots
\end{equation*}

The potential for automatic conjectures is intriguing, but it occurred to us
that \emph{symbolic} experiments could yield automatic proofs rather than
conjectures. These experiments which led to the confirmation of some
conjectures and discoveries of other continued fractions, including three
infinite families.

Before describing our results, we need some notation which deviates slightly
from the norm.

\begin{definition}
    Given two sequences $a(n)$ and $b(n)$ and an integer $m \geq 1$, the
    \emph{$m$th convergent} of their general continued fraction is defined by
    \begin{align*}
              [a(n) : b(n)]_1 &= a(1) \\
        [a(n) : b(n)]_{m + 1} &= a(1) + \frac{b(1)}{[a(n + 1) : b(n + 1)]_m}, \quad m \geq 1,
    \end{align*}
    whenever these expressions are well-defined. If all convergents are
    well-defined and $\lim_{m \to \infty} [a(n) : b(n)]_m$ exists, then the
    \emph{general continued fraction $[a(n) : b(n)]$} is defined as
    \begin{equation*}
        [a(n) : b(n)] = \lim_{m \to \infty} [a(n) : b(n)]_m = a(1) + \frac{b(1)}{a(2) + \frac{b(2)}{a(3) + \cdots}}.
    \end{equation*}
\end{definition}

Classical, \emph{simple} continued fractions are those where $b(n) = 1$ for all
$n$. We shall often make reference to a continued fraction $[a(n) :
b(n)]$ before we have established its existence; in such cases we usually refer
to the convergents or to the sequences $a(n)$ and $b(n)$ themselves.

There is ambiguity in the notation $[a(n) : b(n)]$---what is the sequence
variable?---but we shall always use $n$ as the sequence variable, and other
letters as parameters. For example, the sequences in $[n2^m : 1]$ are $a(n) = n
2^m$ and $b(n) = 1$, not $a(m) = n2^m$ and $b(m) = 1$.

Our principal discovery is the following  doubly-infinite family:
\begin{equation}
    [n + k : an] = \frac{a^D}{(D - 1)! \left( e^a - \sum_{s = 0}^{D - 1} \frac{a^s}{s!} \right)}, \quad D = a + k + 1 \geq 1 \label{linear-family} \\
\end{equation}

We shall get to discovering and proving our results in the following sections,
but first one more historical remark.

As in continued fractions as it is elsewhere, so prolific was Euler that new
results should be checked against his work for duplicates. For example, what is
usually referred to as \emph{Gauss'} continued fraction was discovered by Euler
nearly thirty years before Gauss was born! (See \cite{euler}, or
\cite{eulertrans} for an English translation.) In Section~\ref{sub:euler} we
shall discuss Euler's results in relation to ours and argue that ours are, at
the least, non-trivial to derive from Euler's.

For a more modern introduction to continued fractions, see \cite[ch.~X]{hardy},
and \cite{wall} for the analytic theory in particular.

\section{Experimental continued fractions: The Maple package GCF.txt}
\label{sec:experimental}

This article is accompanied by a Maple package, {\tt GCF.txt} available from
the web-page of this article

\centerline{\url{https://sites.math.rutgers.edu/~zeilberg/mamarim/mamarimhtml/gcf.html},}

where one can also find two sample input and output files with
computer-generated articles for many special cases of our results. Indeed,
while Section~\ref{sec:general} contains human-readable proofs, our results
were \emph{discovered and proved} through symbolic experimentation with
\texttt{GCF.txt}.

This is quite different than ``The Ramanujan Machine'' (TRM) described in
\cite{ramanujan}, both at a high-level and practically. At a high level, TRM
takes a constant and tries to fit a family of continued fractions to it. Our
experiments work in the opposite direction: we construct a family of continued
fractions and try to guess the constants that they generate. While TRM produces
only conjectures, our Maple package produces \emph{proofs}. Of course, the
dazzling conjectures of TRM are motivation for everything in \texttt{GCF.txt}.

Here are short descriptions of the main procedures.

$\bullet$ {\tt GCF(L)}:
Inputs a finite list of pairs of {\bf numbers}, {\tt [[a1,b1],[a2,b2], ..., [ak,bk]]}, outputs
the finite (generalized) continued fraction  it evaluates to.
For example
{\tt GCF([[1,1]\$10]} gives $\frac{89}{55}$ (the tenth convergent to the Golden ratio) while

{\tt GCF([ seq([4*i-2,1],i=1..20) ]);}

gives the $20^{th}$ convergent to Euler's famous continued fraction for $\frac{e+1}{e-1}$.
$$
{\frac{376958612213530151806235679061}{174199042280794948413485144460}} \quad,
$$
that agrees with it to $60$ decimal digits.

$\bullet$ {\tt GCFab(a,b,n,K)}:  Now the input parameters, $a$ and $b$, are {\it expressions} in the symbol $n$.
$K$ is again a positive integer, and the output is the same as {\tt GCF} applied to
$$
[ [a(1),b(1)] , \dots , [a(K),b(K)] ] \quad .
$$

$\bullet$ {\tt RDB(a,b,n)}: Inputs expressions $a$ and $b$ in $n$ and outputs
(if successful) the explicit expressions for the numerator and denominator of
the $n$-th convergent of the infinite continued fraction, as well as its limit,
and the error by using the $50$-th convergent. It uses Maple's {\tt rsolve}
command that is not guaranteed to work (most linear recurrences are not
solvable in closed-form, and even amongst those that are, Maple
[probably] does not know how to
solve all of them). But whenever is succeeds, can be fully trusted, i.e.~it gives a
{\bf proved} result. Under the hood {\tt rsolve} uses Mark Van Hoeij's algorithm
\cite{vanHoeij}, but often the answer can be checked by hand. For instance, in
the case of \eqref{linear-family}, the relevant expressions involve the {\it
incomplete Gamma function}, and can be checked \emph{ab initio} without Maple
by using the incomplete Gamma function's well-known (and easily proved)
first-order inhomogeneous recurrence.

$\bullet$ {\tt Yaron(k1,a1,n,G)}: is a specialization of {\bf RDB}, namely {\tt
RDB(n+k1,a1*n,n)}, but for expository clarity, the incomplete Gamma function
{\tt GAMMA(n+2,-a1)} ($\Gamma(n+2,-a1)$) is denoted by G[n]. It also proves its
results rigorously.

$\bullet$ {\tt YaronV(k1,a1,n,G)}: a verbose form of  {\tt Yaron(k1,a1,n,G)}.
It outputs a computer-generated article.

Note that for each specific pair of sequences, \texttt{RDB} gives the exact
evaluation of the infinite continued fraction, but to prove the {\bf general}
results of our paper, that were conjectured from the many special cases, we had
to `cheat' and use traditional human mathematics. We believe that much of this
human part can also be automated, but leave it to a future paper.

The following section rigorously evaluates an infinite family of general
continued fractions, but there is much more to discover. Our infinite family is
a quite restricted class of general continued fractions generated by specific
linear polynomials. There are other, equally interesting polynomials. For
example, \texttt{RDB} can discover the identity
\begin{equation*}
    [3n : -n(2n - 1)] = \frac{4}{3\pi - 8}.
\end{equation*}
We encourage our readers to experiment and discover new, more exotic,  families.

\section{General proofs}
\label{sec:general}

Our main tool to evaluate continued fractions is exploiting their recursive
nature.

\begin{definition}
    \label{definition:num-denom}
    Given a continued fraction $[a(n) : b(n)]$, define the \emph{numerator} and
    \emph{denominator} sequences $p(n)$ and $q(n)$ by the recurrences
    \begin{align*}
        p(0) &= a(1) \\
        p(1) &= a(1) a(2) + b(1) \\
        p(n + 2) &= a(n + 3) p(n + 1) + b(n + 2) p(n)
    \end{align*}
    and
    \begin{align*}
        q(0) &= 1 \\
        q(1) &= a(2) \\
        q(n + 2) &= a(n + 3) q(n + 1) + b(n + 2) q(n),
    \end{align*}
    respectively.
\end{definition}

\begin{lemma}
    For all positive integers $m$,
    \begin{equation*}
        \frac{p(m)}{q(m)} = [a(n) : b(n)]_{m + 1}.
    \end{equation*}
\end{lemma}

This is a well-known fact from the theory of general continued fractions.

\subsection{First infinite family}
\label{sub:First infinite family}

Our general proof of \eqref{linear-family} relies on some results in
differential equations, so let us first define the necessary objects.

\begin{definition}
    Let
    \begin{equation*}
        M(a, b, z) = \sum_{k \geq 0} \frac{a^{\overline{k}}}{b^{\overline{k}} k!} z^k
    \end{equation*}
    be the \emph{confluent hypergeometric function of the first-kind}---also
    known as \emph{Kummer's function}---and
    \begin{equation*}
        U(a, b, z) = \frac{\Gamma(1 - b)}{\Gamma(a + 1 - b)} M(a, b, z)
                     + \frac{\Gamma(b - 1)}{\Gamma(a)} z^{1 - b} M(a + 1 - b, 2 - b, z)
    \end{equation*}
    be the \emph{confluent hypergeometric function of the second-kind}---also
    known as \emph{Tricomi's function}---where
    \begin{equation*}
        \Gamma(z) = \int_0^\infty e^{-t} t^{z - 1}\ dt
    \end{equation*}
    is the \emph{gamma} function, defined by the above integral whenever $\Re z
    > 0$. Also let
    \begin{equation*}
        \Gamma(z, a) = \int_a^\infty e^{-t} t^{z - 1}\ dt
    \end{equation*}
    be the \emph{incomplete gamma function}.
\end{definition}

The key property of $M(a, b, z)$ and $U(a, b, z)$ is that they are linearly
independent solutions of \emph{Kummer's differential equation}
\begin{equation*}
    z w''(z) + (b - z) w'(z) - aw(z) = 0.
\end{equation*}

Kummer's function $M(a, b, z)$ is entire if $b$ is not a nonpositive integer,
while $U(a, b, z)$ generally has a pole at the origin. In particular, we have
the following well-known result.

\begin{lemma}
    \label{lemma:principle-part}
    If $a - b + 1 = -n$ for a nonnegative integer $n$, then
    \begin{equation*}
        U(a, b, z) = z^{-a} \sum_{s = 0}^n {n \choose s} a^{\overline{k}} z^{-s}.
    \end{equation*}
\end{lemma}

See \cite[ch.~13]{nist} for more details on the confluent hypergeometric functions.

Here is a useful lemma from the world of generating functions.

\begin{lemma}
    \label{lemma:pole}
    Let $f(z)$ be a meromorphic function with a single pole of order $r \geq 2$
    at $z_0 \neq 0$. If
    \begin{equation*}
        f(z) = \sum_{k \geq -r} a_k (z - z_0)^k,
    \end{equation*}
    then
    \begin{equation*}
        [z^n] f(z) = \frac{(-1)^r a_{-r}}{z_0^{n + r}} {n + r - 1 \choose r - 1}(1 + O(1 / n)).
    \end{equation*}
    where $[z^n] f(z)$ is the coefficient on $z^n$ in the expansion of $f$
    about the origin.
\end{lemma}

\begin{proof}
    Let
    \begin{equation*}
        g(z) = \sum_{-r \leq k < 0} a_k (z - z_0)^k
    \end{equation*}
    be the principle part of $f$ at $z_0$. Expressing $g(z)$ as a power
    series about the origin in the usual way yields
    \begin{equation*}
        [z^n] g(z) = \sum_{k = 1}^r \frac{(-1)^k a_{-k}}{z_0^{k + n}} {n + k - 1 \choose k - 1}.
    \end{equation*}
    In particular, if $r \geq 2$, then this is a polynomial in $n$ of degree $r
    - 1 \geq 1$, and $n^{r - 1}$ only appears in the last term. Since $f(z) -
    g(z)$ is entire, $[z^n] (f(z) - g(z)) = O(1)$, and pulling out the leading
    term after rearranging yields
    \begin{equation*}
        [z^n] f(z) = \frac{(-1)^r}{z_0^{n + r}} {n + r - 1 \choose r - 1}(1 + O(1 / n))
    \end{equation*}
    as desired.
\end{proof}

Using the previous lemma, we will now provide asymptotic expansions for $p(n)$
and $q(n)$ for the continued fraction $[n + k : an]$. The key observation is
that the exponential generating functions of $p(n)$ and $q(n)$ both satisfy the
same second-order differential equation, and that this equation has a nice,
meromorphic solution with a single pole at $z = 1$. To avoid repetition, let's
first prove a very specific lemma.

\begin{lemma}
    \label{lemma:kummer-asymp}
    Let $a(n)$ be a sequence whose exponential generating function $A(z) =
    \sum_{n \geq 0} \frac{a(n)}{n!} z^n$ satisfies
    \begin{equation*}
        A(z) = L e^{-\alpha z} M(D, D + 2, \alpha (z - 1))
                    + E e^{-\alpha z} U(D, D + 2, \alpha (z - 1))
    \end{equation*}
    for some values $L$ and $E$ independent of $z$, a positive integer $D$, and
    a nonzero real $\alpha$. Then
    \begin{equation*}
        a(n) = E \frac{e^{-\alpha}}{\alpha^{D + 1}} D {n + D \choose D} (1 + O(1 / n)).
    \end{equation*}
\end{lemma}

\begin{proof}
    Kummer's hypergeometric function is entire, while $U(D, D + 2, \alpha (z -
    1))$ has a single pole at $z = 1$. In fact, by
    Lemma~\ref{lemma:principle-part}, the pole is order $D + 1 \geq 2$, and the
    coefficient on its lowest degree term is $D / \alpha^{D + 1}$. If we write
    \begin{equation*}
        E e^{-\alpha z} U = e^{-\alpha} E e^{-\alpha (z - 1)} U,
    \end{equation*}
    then we see that the coefficient on the lowest degree term of $E e^{-\alpha
    z} U$ is $E e^{-\alpha} D / \alpha^{D + 1}$, so Lemma~\ref{lemma:pole}
    implies $a(n) = E \frac{e^{-\alpha}}{\alpha^{D + 1}} D {n + D \choose D} (1
    + O(1 / n))$.
\end{proof}

\begin{theorem}
    Let $a$ and $k$ be integers. If $D = a + k + 1 \geq 1$, then
    \begin{equation*}
        [n + k : an] = \frac{a^D}{(D - 1)! \left( e^a - \sum_{s = 0}^{D - 1} \frac{a^s}{s!} \right)},
    \end{equation*}
    provided that the convergents of the continued fraction are well-defined.
\end{theorem}

\begin{proof}
    Let
    \begin{align*}
        P(z) &= \sum_{n \geq 0} \frac{p(n)}{n!} z^n \\
        Q(z) &= \sum_{n \geq 0} \frac{q(n)}{n!} z^n
    \end{align*}
    be the exponential generating functions of $p(n)$ and $q(n)$, respectively.
    By well-known facts about egfs, the recurrences in
    Definition~\ref{definition:num-denom} imply that $P$ and $Q$ both satisfy
    the second-order differential equation
    \begin{equation*}
        (1 - z) f''(z) - (az + k + 3) f'(z) - 2a f(z) = 0,
    \end{equation*}
    with initial conditions
    \begin{alignat*}{2}
        P(0) &= k + 1 \quad &&P'(0) = (k + 1)(k + 2) + a \\
        Q(0) &= 1 \quad &&Q'(0) = k + 2.
    \end{alignat*}
    This reduces to a special case of \emph{Kummer's equation}. It is easy to
    check with a computer algebra system that the general solution is
    \begin{equation*}
        f(z) = A(k) e^{-az} M(D, D + 2, a(z - 1)) + B(k) e^{-az} U(D, D + 2, a(z - 1))
    \end{equation*}
    for some sequences $A(k)$ and $B(k)$ which depend on the initial
    conditions. By Lemma~\ref{lemma:kummer-asymp}, it suffices to compute
    $B(k)$ for $p(n)$ and $q(n)$ separately.

    Let $B_p(a, k)$ be the coefficient on $e^{-az} U$ for $P(z)$, and $B_q(a,
    k)$ the coefficient on $e^{-az} U$ for $Q(z)$. We may compute these
    functions by solving the relevant initial condition equations. For
    instance,
    \begin{equation*}
        B_p(a, k) = (-1)^{a + k} a^{a + k + 2} = (-a)^{D + 1}.
    \end{equation*}
    The function $B_q(a, k)$ is significantly more complicated, but still
    routine to compute. After some coercing, Maple simplifies it as
    \begin{align*}
        B_q(a, k) &= \frac{a(e^a(\Gamma(D + 1, a) - \Gamma(D + 1)) a^{D + 1} - a^{2D + 1}) (-1)^D}{Da^{D + 1}} \\
                  &= (-1)^D \frac{a}{D} (e^a (\Gamma(D + 1, a) - \Gamma(D + 1)) - a^D).
    \end{align*}
    The incomplete gamma function expression can be written
    \begin{equation*}
        \Gamma(D + 1, a) - \Gamma(D + 1) = -\int_0^a e^{-t} t^D\ dt
                                         = D! \left( e^{-a} \sum_{s = 0}^D \frac{a^s}{s!} - 1 \right),
    \end{equation*}
    so
    \begin{align*}
        B_q(a, k)
            &= (-1)^D \frac{a}{D} \left( D! \left( \sum_{s = 0}^D \frac{a^s}{s!} - e^a \right) - a^D \right) \\
            &= (-1)^D a (D - 1)! \left(\sum_{s = 0}^{D - 1} \frac{a^s}{s!} - e^a \right).
    \end{align*}

    Putting everything together, we have
    \begin{align*}
        [n + k : an] &= \lim_{m \to \infty} \frac{p(m)}{q(m)} \\
                     &= \lim_{m \to \infty} \frac{B_p(a, k) (1 + O(1 / m))}{B_q(a, k) (1 + O(1 / m))} \\
                     &= \frac{B_p(a, k)}{B_q(a, k)} \\
                     &= \frac{a^D}{(D - 1)! \left( e^a - \sum_{s = 0}^{D - 1} \frac{a^s}{s!} \right)},
    \end{align*}
    as claimed.
\end{proof}

There are some notable special cases. If $a = -k$, then $D = 1$, which yields
\begin{equation*}
    [n + k : -kn] = \frac{ke^k}{e^k - 1}.
\end{equation*}

If $a = -1$, then $D = k$, and we can write
\begin{equation*}
    [n + k + 1: -n] = \frac{(-1)^k e}{e k! \sum_{s = 0}^k \frac{(-1)^s}{s!} - k!}
\end{equation*}
for nonnegative integers $k$. Equation~\ref{example-frac} is then obtained with
$k = 2$:
\begin{equation*}
    [n + 3 : -n] = \frac{e}{e - 2}.
\end{equation*}
Note that
\begin{equation*}
    e k! \sum_{s = 0}^k \frac{(-1)^s}{s!} = e[k! / e],
\end{equation*}
where $[x]$ denotes the integer nearest to the real $x$. This is a remarkable
coincidence, since $[k! / e]$ is the $k$th \emph{derangement number}, the
number of permutations on $k$ objects with no fixed points. There does not seem
to be any immediate combinatorial reason for the derangement numbers to appear.

\subsection{Second infinite family}
\label{sub:Second infinite family}

Our second infinite family  may be derived from Euler's general continued fraction
but we prefer to discover it {\it ab initio}, our way.
Note that our approach gives much more than the limit, it gives closed-form expressions for the numerator
and denominator of the truncated sequence.

\begin{theorem}
    Let $a$ and $b$ be nonnegative reals such that $a \neq 0$. Then
    \begin{equation*}
        [an^2 + bn + 1 : -an^2 - bn]
            = \frac{4 F(a, b) + 2(2a + b)(a + b + 1))}{4 F(a, b) + 2(2a + b)},
    \end{equation*}
    where
    \begin{equation*}
        F(a, b) = 2 \sum_{k \geq 0} \frac{1}{(k + 2)! (3 + b / a)^{\overline{k}} a^k}.
    \end{equation*}
\end{theorem}

\begin{proof}
    The implied recurrence numerator and denominator recurrences can be solved
    and easily put into asymptotic form. The solutions are, asymptotically,
    \begin{align*}
        p(n) &= C(a, b, n) ((F(a, b) / 2 + (2 a + b)(a + b + 1)) + O(1 / n^2)) \\
        q(n) &= \frac{1}{2} C(a, b, n) ((F(a, b) + 2(2a + b)) + O(1 / n^2)),
    \end{align*}
    where $C(a, b, n)$ is some function. From this, it is easy to see that
    \begin{equation*}
        \lim_{n \to \infty} \frac{p(n)}{q(n)} = \frac{F(a, b) + 2(2a + b)(a + b + 1)}{F(a, b) + 2(2a + b)}. \qedhere
    \end{equation*}
\end{proof}

The $F$ function is actually a special case of the general hypergeometric
function, and therefore offers numerous opportunities for closed-form
evaluation. For example, suppose that $b / a = m - 1 / 2$ for some nonnegative
integer $m$. Then, from the identity
\begin{equation*}
    (r - 1 / 2)^{\overline{k}} = \frac{(2r - 1)^{\overline{2k}}}{4^k r^{\overline{k}}},
\end{equation*}
we have
\begin{align*}
    (3 + b / a)^{\overline{k}} &= (3 + m - 1 / 2)^{\overline{k}} \\
                               &= \frac{(5 + 2m)^{\overline{2k}}}{4^k (3 + m)^{\overline{k}}} \\
                               &= \frac{(4 + 2m + 2k)! (2 + m)!}{4^k (4 + 2m)! (2 + m + k)!}.
\end{align*}
Therefore
\begin{equation*}
    F(a, b) = \frac{2(4 + 2m)!}{(m + 2)!} \sum_{k \geq 0} \frac{(k + m + 2)!}{(k + 2)! (2k + 2m + 4)!} \left( \frac{4}{a} \right)^k.
\end{equation*}

This remaining sum looks quite burly, but is amenable to routine evaluation
after some simplifications. Let us give a brief sketch of how it might work.

In what follows, let us write ``$\sim$'' to mean ``equal up to a multiplicative
constant.'' First, shifting the summation index by $m + 2$ gives
\begin{equation*}
    F(a, b) \sim \sum_{k \geq 2} \frac{(k + m)!}{k!(2k + 2m + 1)!} \left( \frac{4}{a} \right)^k.
\end{equation*}
Note that $(k + m)! / k! = (k + m)^{\underline{m}}$, so
\begin{equation*}
    F(a, b) \sim \sum_{k \geq 2} \frac{(k + m)^{\underline{m}}}{(2k + 2m + 1)!} \left( \frac{4}{a} \right)^k.
\end{equation*}
Now shifting the index by $m$ gives
\begin{equation*}
    F(a, b) \sim \sum_{k \geq m + 2} \frac{k^{\underline{m}}}{(2k + 1)!} \left( \frac{4}{a} \right)^k.
\end{equation*}
At this point we have won, because the series
\begin{equation*}
    \sum_{k \geq m + 2} \frac{1}{(2k + 1)!} z^k
\end{equation*}
is known, and $k^{\underline{m}}$ is a polynomial in $k$.

More explicitly, set
\begin{equation*}
    f(z) = \sum_{k \geq m + 2} \frac{1}{(2k + 1)!} z^k,
\end{equation*}
and note that $f(z)$ is $z^{-1 / 2} \sinh \sqrt{z}$ minus a finite number of
initial terms. From the elementary theory of generating functions, since
$k^{\underline{m}}$ is a polynomial in $k$, we have
\begin{equation*}
    \sum_{k \geq m + 2} \frac{k^{\underline{m}}}{(2k + 1)!} z^k
    =
    (zD)^{\underline{m}} f(z),
\end{equation*}
where $D$ is the differentiation operator $Df = f'$. Since we know $f$, we can
carry out the iterated differentiations and then set $z = 4 / a$ to obtain an
answer. Note that the hyperbolic trigonometric functions are closed under
differentiation, so our answer will be in terms of them. The full result is too
messy to completely record, but these routine operations can be completed by
any computer algebra system.

For example, if $a = 4$ and $b = 6$, then following the above steps will
eventually produce
\begin{align*}
    F(4, 6) &= -308 + 840(\cosh(1) - \sinh(1)) \\
            &= -308 + \frac{840}{e}.
\end{align*}
In this case we obtain
\begin{align*}
    [4n^2 + 6n + 1 : -4n^2 - 6n] &= \frac{840 / e}{840 / e - 280} \\
                                 &= \frac{3}{3 - e}.
\end{align*}
Taking $a = 6$ and $b = 9$, we obtain
\begin{equation*}
    [6 n^2 + 9 n + 1 : -6 n^2 - 9 n] =
        \frac{-9\,\sqrt{6}\sinh(\sqrt{6} / 3) + 18 \cosh(\sqrt{6} / 3)}
             {-9\sqrt{6}\sinh(\sqrt{6} / 3) + 18 \cosh(\sqrt{6} / 3) - 4}.
\end{equation*}

There are likely many other nice cases.

\subsection{Third infinite family}
\label{sub:Third infinite family}

Our third infinite family is an immediate consequence of Euler's, but once again, we
do it {\it our} way.
Again note that our approach gives much more than the limit, it gives closed-form expressions for the numerator
and denominator of the truncated sequence (but we admit that in this case it is fairly straightforward).

\begin{theorem}
    For integers $k \geq 2$,
    \begin{equation*}
        [(n - 1)^k + n^k, -n^{2k}] = \frac{1}{\zeta(k)}.
    \end{equation*}
\end{theorem}

\begin{proof}
    It is routine to check that $p(n) = (n + 1)!^k$ is the numerator sequence
    for this continued fraction. It is also routine to check (but difficult to
    discover) that the denominator sequence of the continued fraction is
    \begin{equation*}
        q(n) = (n + 1)!^k \left( \frac{\psi(k, n + 2)}{(k - 1)!} + \zeta(k) \right),
    \end{equation*}
    where $\psi(k, z)$ is the $k$th \emph{polygamma function}, which may be
    defined by
    \begin{equation*}
        \psi(k, z) = (-1)^{k + 1} k! \sum_{j \geq 0} \frac{1}{(z + j)^{k + 1}}
    \end{equation*}
    for $z \notin \{-1, -2, \dots\}$. This gives
    \begin{equation*}
        \frac{p(n)}{q(n)} = \frac{1}{\zeta(k)} \frac{1}{\psi(k, n + 2) O(1) + 1},
    \end{equation*}
    where the $O(1)$ is some constant independent of $n$. It is easy to check
    with the dominated convergence theorem that $\psi(k, n + 2) \to 0$ as $n
    \to \infty$ for all $k \geq 2$, which implies
    \begin{equation*}
        \lim_{n \to \infty} \frac{p(n)}{q(n)} = \frac{1}{\zeta(k)}
    \end{equation*}
    as claimed.
\end{proof}

As a demonstration, $k = 3$:
\begin{equation*}
    \frac{1}{\zeta(3)} = 1 - \frac{1|}{|9} - \frac{64|}{35|} - \frac{729|}{|91} + \cdots
\end{equation*}
Just using the terms listed, we have:
\begin{align*}
    1 - \frac{1|}{|9} - \frac{64|}{35|} - \frac{729|}{|91} = \frac{1728}{2035} &\approx 0.84914004914004914 \\
    \frac{1}{\zeta(3)} &\approx 0.83190737258070746.
\end{align*}
Not a great approximation, but it is something.

\subsection{The method of Euler}
\label{sub:euler}

Our method of proving continued fractions is relatively direct and simple:
write down the ``obvious'' recurrences, solve them, then take a limit. Another
approach is to combine some specializations of well-known infinite families
with manipulation techniques. This situation is analogous to the problems of
evaluating sums or intengrals---try some standard techniques, or attempt to
contort your problem until it matches some well-known, general result.

One such general result for continued fractions is \emph{Euler's continued
fraction} we mentioned in the introduction:
\begin{equation*}
    \cfrac{1}{1 - \cfrac{r_1}{1 + r_1 - \cfrac{r_2}{1 + r_2 - \cfrac{r_3}{1 + r_3 - \cdots}}}}
    = \sum_{k \geq 0} \prod_{j = 1}^k r_j.
\end{equation*}
Slight manipulation yields
\begin{equation*}
    \cfrac{-r_1}{1 + r_1 - \cfrac{r_2}{1 + r_2 - \cfrac{r_3}{1 + r_3 - \cdots}}}
    = \frac{1}{\sum_{k \geq 0} \prod_{j = 1}^k r_j} - 1,
\end{equation*}
or, in our notation,
\begin{equation*}
    [r(n - 1) + 1 : -r(n)] = r(0) + \frac{1}{\sum_{k \geq 0} \prod_{j = 1}^k r(j)}.
\end{equation*}

Euler's continued fraction directly applies to situations where ``numerator $+$
denominator $= 1$.'' (In our notation this is ``numerator $+$ \emph{previous}
denominator $= 1$.'') None of our families are of this form, but Euler's reach
can be expanded with an \emph{equivalence transformation}. Namely, it is easy
to prove by induction that
\begin{equation*}
    [a(n) : b(n)] = [c(n - 1) a(n) : c(n - 1) c(n) b(n)]
\end{equation*}
for any strictly nonzero sequence $\{c_n\}$ with $c(0) = 1$, provided that both
continued fractions exist.

In light of this, one way to evaluate $[a(n) : b(n)]$ is to find sequences
$r(n)$ and $c(n)$ such that
\begin{align*}
    a(n) &= c(n - 1) (r(n - 1) + 1) \\
    b(n) &= -c(n - 1) c(n) r(n),
\end{align*}
for $n \geq 1$. Once these sequences are in hand, we can write
\begin{align*}
    [a(n) : b(n)] &= [c(n - 1)(r(n - 1) + 1) : -c(n - 1) c(n) r(n)] \\
                  &= [r(n - 1) + 1 : -r(n)] \\
                  &= r(0) + \left( \sum_{k \geq 0} \prod_{j = 1}^k \frac{b(j)}{a(j + 1) + b(j)} \right)^{-1}.
\end{align*}
Our problem reduces to computing the remaining sum.

It is not obvious how to find such sequences. For instance, take $a = -1$ and
$k = 1$ in \eqref{linear-family}. Then we must solve the equations
\begin{align*}
    n + 1 &= c(n - 1) (r(n - 1) + 1) \\
    -n &= -c(n - 1) c(n) r(n)
\end{align*}
for $c(n)$ and $r(n)$. It is easy to find that
\begin{equation*}
    r(n) = \frac{n}{c(n - 1)n - n + 2 c(n - 1)},
\end{equation*}
but solving for $c(n)$ is not so easy. It is possible that some suitable
\emph{ansatz} could let us \emph{guess} $c(n)$, but this is another question
entirely.

\section{Conclusion}
\label{sec:Conclusion}

We have given a Maple package and a new and interesting doubly-infinite family
thereby proving  conjectures made in \cite{ramanujan} (and considerably generalizing them).
More important, we have illustrated a {\it methodology} of computer-assisted discovery and proof
that, however, still needs some human intervention for the general case, but would have been impossible
without the many special cases proved by the computer, enabling us humans to generalize it.
We believe that this last step would soon also be fully automated.

The authors would like to thank the OEIS \cite{oeis} for its help in
identifying promising sequences of integers that arose in the evaluation of
these continued fractions.

\end{document}